\documentclass[a4paper, 11pt, nosumlimits, twoside] {amsart}
\title[Lazard Isomorphism]{A comparison of locally analytic group cohomology and Lie algebra cohomology for $p$-adic Lie groups}
\author{Sabine Lechner}

\thanks{These results are part of my Ph.D thesis and I would like to thank my adviser Annette Huber-Klawitter for her supervision and her inspiring support. The thesis was made possible by grants of DFG Forschergruppe "Algebraische Zykel und L-Funktionen" in Freiburg. }

\usepackage{amsmath}
\usepackage{mathrsfs}  
\usepackage{amssymb}
\usepackage{amsfonts}
\usepackage{amsthm}
\usepackage[latin1]{inputenc}
\usepackage{graphicx}
\usepackage[T1]{fontenc} 
\usepackage[arrow,curve,matrix]{xy}

\numberwithin{subsection}{section}
\DeclareMathOperator{\sgn}{sgn}
\DeclareMathOperator{\Der}{Der}

\DeclareMathOperator{\Hom}{Hom}

\DeclareMathOperator{\Ind}{Ind}

\DeclareMathOperator{\Z}{\mathbb{Z}}
\DeclareMathOperator{\Q}{\mathbb{Q}}
\DeclareMathOperator{\Gc}{\mathcal{G}}
\DeclareMathOperator{\Hc}{\mathcal{H}}
\DeclareMathOperator{\Gch}{\mathcal{G}_h}
\DeclareMathOperator{\g}{\mathfrak{g}}
\DeclareMathOperator{\mal}{\bullet}
\DeclareMathOperator*{\OG}{\mathcal{O}(\textit{G})}
\DeclareMathOperator*{\OGK}{\mathcal{O}(\textit{G}_{\textit{K}})}
\DeclareMathOperator*{\tOG}{\tilde{\mathcal{O}}(\textit{G})}
\DeclareMathOperator*{\tOGK}{\tilde{\mathcal{O}}(\textit{G}_{\textit{K}})}
\DeclareMathOperator*{\OGe}{\mathcal{O}^{la}(\mathcal{G}(0))_e}

\DeclareMathOperator*{\Ug}{\mathcal{U}(\mathfrak{g})}
\DeclareMathOperator*{\U*}{\mathcal{U}^\star}

\DeclareMathOperator*{\D}{\mathcal{D}}
\DeclareMathOperator*{\als}{\beta^\star}
\DeclareMathOperator*{\tals}{\tilde{\beta}^\star}
\DeclareMathOperator*{\tphi}{\tilde{\phi}}
\DeclareMathOperator*{\tK}{\tilde{K}}
\DeclareMathOperator*{\al}{\beta}
\DeclareMathOperator*{\V}{V(\mathfrak{g})}

\newcommand{\bes}{\begin{eqnarray*}}
\newcommand{\ees}{\end{eqnarray*}}
\newcommand{\be}{\begin{eqnarray}}
\newcommand{\ee}{\end{eqnarray}}

\mathchardef\ordinarycolon\mathcode`\:				
  \mathcode`\:=\string"8000
  \begingroup \catcode`\:=\active
    \gdef:{\mathrel{\mathop\ordinarycolon}}
  \endgroup

\newtheoremstyle{note}
  {}
  {}
  {}
  {}
  {\bfseries}
  {.}
  { }
  {}

\newtheoremstyle{notes}
  {}
  {}
  {\itshape}
  {}
  {\bfseries}
  {.}
  { }
  {}

\newtheoremstyle{numberlessthm}{}{}{\itshape}{}{\bfseries}{.}{ }{\thmnote{#3}}
\newtheoremstyle{numberlessnote}{}{}{}{}{}{}{ }{\thmnote{#3}}

\theoremstyle{notes}					

\newtheorem{theorem}{Theorem}[section]
\newtheorem{prop}[theorem]{Proposition}
\newtheorem{lemma}[theorem]{Lemma}
\newtheorem{thm}[theorem]{Theorem}

\newtheorem{cor}[theorem]{Corollary}

\theoremstyle{note}						

\newtheorem{defn}[theorem]{Definition}
\newtheorem{defnprop}[theorem]{Definition and Proposition}

\newtheorem{ex}[theorem]{Example}

\newtheorem{rem}[theorem]{Remark.}

\newtheorem{notation}[theorem]{Notation.}

\theoremstyle{numberlessthm}			

\theoremstyle{numberlessnote}			

\begin{document}

\begin{abstract}
The main result of this work is a new proof and generalization of Lazard's comparison theorem of locally analytic group cohomology with Lie algebra cohomology for $K$-Lie groups, where $K$ is a finite extension of $\Q_p$.
\end{abstract}
\maketitle

\tableofcontents

\section{Introduction}
Lazard's comparison theorem was one of the main results in his work on $p$-adic groups \cite{lazard}. It relates locally analytic group cohomology with Lie algebra cohomology for $\Q_p$-Lie groups in two steps. First Lazard worked out an isomorphism between locally analytic group cohomology and continuous group cohomology and secondly between continuous group cohomology and Lie algebra cohomology. The latter is obtained from a difficult isomorphism between the saturated group ring and the saturated universal enveloping algebra.

Huber and Kings showed in \cite{HK} that one can directly define a map from locally analytic group cohomology to Lie algebra cohomology by differenting cochains and that in the case of smooth algebraic group schemes $H$ over $\Z_p$ with formal group $\mathcal{H} \subset H(\Z_p)$ the resulting map $$ \Phi: H^n_{la}(\mathcal{H}, \Q_p) \rightarrow H^n(\mathfrak{h}, \Q_p)$$ coincides, after identifying continuous group cohomology with locally analytic group cohomology, with Lazard's comparison isomorphism (\cite[Theorem 4.7.1]{HK}). Serre mentioned to the aforementioned authors that this was clear to him at the time Lazard's paper was written, however it was not included in the published results. In their joint work with N. Naumann in \cite{HKN} they extended the comparison isomorphism for $K$-Lie groups attached to smooth group schemes with connected generic fibre over the integers of $K$ (\cite[Theorem 4.3.1]{HKN}). The aim of this thesis is to use this simpler map to obtain an independent proof of Lazard's result.

The context in which Huber and Kings worked out the new description of the Lazard isomorphism is the construction of a $p$-adic regulator map in complete analogy to Borel's regulator for the infinite prime. The van Est isomorphism between relative Lie algebra cohomology and continuous group cohomology is replaced by the Lazard isomorphism. Their aim is to use this construction of a $p$-adic regulator for attacking the Bloch-Kato conjecture for special values of Dedekind Zeta functions.

Our strategy to prove the comparison isomorphism between locally analytic group cohomology and Lie algebra cohomology is to trace it back to the case of a formal group law $G$. Hence the first step is to obtain an isomorphism of formal group cohomology with Lie algebra cohomology (Corollary \ref{cor1})
$$\tilde{\Phi}: H^n(\tilde{G}, R) \rightarrow H^n(\g,R),$$ where $R$ is an integral domain of characteristic zero. The tilde over $\Phi$ and over the formal group law $G$ indicate that one has to modify the formal group cohomology while working with coefficients in $R$. This means that if the ring of functions to $G$, called $\OG$, is given by a formal power series ring over $R$ one has to allow certain denominators. However, we will prove in Lemma \ref{convlemma} that functions of this modified ring of functions $\tOG$ still converge, with the same region of convergence as the exponential function. \\
Let $\g$ be the Lie algebra associated to the formal group law $G$ and $\Ug$ its universal enveloping algebra. Then an essential ingredient to the first step is a morphism of complete Hopf algebras (Proposition \ref{als}) $$\als: \OG \rightarrow \Ug$$ from the ring of functions to the universal enveloping algebra. We show in Proposition \ref{gen} that this morphism is an isomorphism if we consider the modified ring of functions $\tOG$. Furthermore, we prove in Theorem \ref{phi} that this isomorphism extends to a quasi-isomorphism of the corresponding complexes and hence to the above isomorphism $\tilde{\Phi}$. \\
The second step in the proof of the Main Theorem is a Comparison Theorem for standard $K$-Lie groups. These standard groups are $K$-Lie groups associated to a formal group law $G$, see Definition \ref{g1}. 

\textbf{Comparison Theorem for standard groups} [Thm.\ref{thmstd}]. \textit{Let $G$ be a formal group law over $R$ and let $\Gc(h)$ be the $m$-standard group of level $h$ to $G$ with Lie algebra $\g \otimes_R K$. Then the map \begin{align*} \Phi_s: H^n_{la}(\Gc(h), K) \rightarrow H^n(\g, K) \end{align*} given by the continuous extension of \begin{align*} f_1 \otimes \cdots \otimes f_n \mapsto df_1\wedge \cdots \wedge df_n, \end{align*}
for $n\geq 1$ and by the identity for $n=0$ is an isomorphism for all $h~>~h_0=~\frac{1}{p-1}$.} 

The Main Theorem can then be deduced from the Comparison Theorem for standard $K$-Lie groups, since every $K$-Lie group contains an open subgroup, which is standard, see Lemma \ref{subgr}. Our approach to the proof of the Comparion Theorem for standard $K$-Lie groups is as follows. In a first step, we will show using the isomorphism $\tilde{\Phi}$ that the limit morphism $$\Phi_\infty: H^n(\mathcal{O}^{la}(\mathcal{G}(0)^\bullet)_e, K) \rightarrow H^n(\mathfrak{g}, K),$$ associated to the ring of germs of locally analytic functions in $e$, denoted by $\mathcal{O}^{la}(\Gc(0))_e$, is an isomorphism. Then injectivity of $\Phi_s$ follows from a spectral sequence argument. The proof of this injectivity part will be analogous to the proof of Theorem 4.3.1 in \cite{HKN}, however independent of the work of Lazard \cite{lazard}. For surjectivity we will again use the isomorphism statement for formal group cohomology of Corollary \ref{cor1} in addition to the aforementioned fact that functions of $\tOG$ still converge.

The article is organized as follows: In Section 2 we give a number of definitions and well-known facts concerning formal groups, Lie algebras, Hopf algebra structures and cohomology complexes. Section 3 deals, firstly, with the existence of a morphism of complete Hopf algebras between the ring of functions of a formal group law and the dual of the universal enveloping algebra. Secondly we consider the cases where this morphism is an isomorphism and thirdly, considering the modified ring of functions, we can prove that the morphism is in this modified case actually an isomorphism of complete Hopf algebras. In Section 4 we show that the isomorphism of Section 3 can be extended to a quasi-isomorphism of the corresponding complexes and we will give an explicit description, which will be identical to the description of the comparison map in \cite{HKN} and hence to Lazard's map. The last Section 5 gives the proof of the Main Theorem. We will begin by fixing some notation in Subsection 5.1, in order to formulate the Comparison Theorem for standard groups. The proof of this theorem will be given in the remaining two subsections.

\section{Preliminaries and Notation}
The main objects we are dealing with are introduced in this section. We denote by $R$ an integral domain of characteristic zero. 
\subsection{Formal group law, Lie algebra and Universal enveloping algebra} \label{sec1}

\begin{defn} \cite[9.1 Definitions]{haze}\label{law} 
Let $\textbf{X}=(X_1, \ldots, X_m)$ and $\textbf{Y}=(Y_1,\ldots, Y_m)$ be two sets of $m$ variables. An \textit{$m$-dimensional formal group law over $R$} is an $m$-tuple of power series $$G(\textbf{X},\textbf{Y})=(G^{(1)}(\textbf{X},\textbf{Y}), \ldots , G^{(m)}(\textbf{X},\textbf{Y}))$$ with $G^{(j)}(\textbf{X},\textbf{Y}) \in R[[\textbf{X},\textbf{Y}]]$ such that for all $j=1,\ldots , m$
\begin{itemize}
\item[(1)] $G^{(j)}(\textbf{X},\textbf{Y}) = X_j+ Y_j+ \sum_{l,k=1}^{m} \gamma_{lk}^j X_lY_k + O(d \geq 3), \gamma_{lk}^j \in R$
\item[(2)] $G^{(j)}(G(\textbf{X},\textbf{Y}),\textbf{Z}) = G^{(j)}(\textbf{X},G(\textbf{Y},\textbf{Z}))$,
\end{itemize}
where the notation $O(d \geq n)$ stands for a formal power series whose homogeneous parts vanish in degree strictly less than $n$. If in addition $$G^{(j)}(\textbf{X},\textbf{Y})=G^{(j)}(\textbf{Y},\textbf{X})$$ holds for all $j=\{1,\ldots , m\}$, then the formal group law is called commutative.
The \textit{ring of functions} $\OG$ to a formal group law $G$ is the ring of formal power series in $m$ variables $t^{(1)}, \ldots, t^{(m)}$ over $R$, i.e. $\mathcal{O}(G)=R[t^{(1)}, \ldots, t^{(m)}]$.\end{defn}

\begin{prop} \cite[Appendix A.4.5]{haze} Let $G(\textbf{X},\textbf{Y})$ be an $m$-dimensional formal group law over $R$. Then there exists a power series $s(\textbf{X})$ such that $G(\textbf{X}, s(\textbf{X}))=0.$ \end{prop} 

\begin{rem} \label{exs}  The proof of the existence of the power series $s(\textbf{X})$ gives an explicit construction of this power series. The first step of the construction yields $s(\textbf{X})= - \textbf{X}$ mod (degree $2$), a fact we will need later in Section \ref{sechom}. \end{rem}

\begin{defn} \label{hom}\cite[9.4 Homomorphisms and isomorphisms]{haze} Let $G(\textbf{X},\textbf{Y})$ and $G'(\textbf{X},\textbf{Y})$ be $m$-dimensional formal group laws over $R$. A \textit{homomorphism} $$G(\textbf{X},\textbf{Y}) \rightarrow G'(\textbf{X},\textbf{Y})$$ over $R$ is an $m$-tuple of power series $\alpha(\textbf{X})$ in $n$ indeterminantes such that $\alpha(\textbf{X}) \equiv 0$ mod (degree 1) and $$ \alpha(G(\textbf{X},\textbf{Y}))= G'(\alpha(\textbf{X}),\alpha(\textbf{Y})).$$ The homomorphism $\alpha(\textbf{X})$ is an \textit{isomorphism} if there exists a homomorphism $\beta(\textbf{X}): G'(\textbf{X},\textbf{Y}) \rightarrow G(\textbf{X},\textbf{Y})$ such that $\alpha(\beta(\textbf{X}))= \beta(\alpha(\textbf{X}))$. \end{defn}

\begin{lemma} \label{filtog}The ring of functions $\OG$ to a formal group law $G$ is complete with respect to the topology induced by the following descending filtration 
\begin{eqnarray} 
F^i\OG=\left\lbrace f \in \OG | \mbox{all monomials of} \ f \ \mbox{have total degree} \geq i \right\rbrace.\end{eqnarray} \end{lemma}

\begin{notation} \label{not} \begin{itemize} 
\item We denote by the sign $\hat{\otimes}$ the completed tensor product with respect to the above topology. Thus we can identify  $\OG^{\hat{\otimes}n}$ with the ring of formal power series in $nm$ indeterminates $$R[[t_1^{(1)}, \ldots t_1^{(m)},t_2^{(1)}, \ldots, t_2^{(m)}, \ldots,t_n^{(1)}, \ldots, t_n^{(m)} ]].$$
\item For the elements $t_i^{(j)}$ of the ring of functions $\OG$ we will use equivalently the notation $1 \otimes \ldots \otimes 1 \otimes t^{(j)} \otimes 1 \otimes \ldots \otimes 1$ (with $t^{(j)}$ at the $i$-th entry) for all $j=1,\ldots , m$. 
\item For simplicity we write \begin{itemize} 
\item $\underline{t}_1$ for $t_1^{(1)}, \ldots, t_1^{(m)}$, 
\item $\underline{t}_{1, \ldots, n}$ for $t_1^{(1)}, \ldots, t_1^{(m)},t_2^{(1)}, \ldots, t_2^{(m)}, \ldots, t_n^{(1)}, \ldots, t_n^{(m)}.$  \end{itemize}
\item We use the general multi-index notation $\underline{j}$ for the tuple $(j_1, \ldots ,j_m)$. 
\item If $m=1$ or $n=1$ we skip the upper, respectively lower index and write $t_i$ for $t_i^{(1)}$ respectively $t^{(j)}$ for $t^{(j)}_1$. 
\end{itemize}
\end{notation}

\textbf{Lie algebra and universal enveloping algebra} \\
Let $R[\underline{t}_1]$ be the polynomial ring over $R$ in $m$ variables and let $\frac{\partial}{\partial t_1^{(j)}}$ for all $j \in \{1,\ldots, m\}$ be the $j$-th partial derivative. Let $\Der_R(\OG, \OG)$ denote the set of $R$-derivations of $\OG$. Then $\Der_R(\OG, \OG)$ is a free $\OG$-module on the basis $\frac{\partial}{\partial t_1^{(1)}}, \ldots, \frac{\partial}{\partial t_1^{(m)}}$. If $d, d_1, d_2 \in  \Der_R(\OG, \OG)$ and $r \in R$, then the mappings $rd, d_1+d_2$ and $[d_1, d_2]:= d_1d_2-d_2d_1$ are also derivations, see \cite[Chap.III §10.4]{bour1}. Thus, the set $\Der_R(\OG, \OG)$ is an $R$-module which is also a Lie algebra. 

Let $e_j$ denote the $j$-th partial derivative $\frac{\partial}{\partial t_1^{(j)}}$ evaluated at $0$. We denote by $\g$ the free $R$-module on the basis $e_1, \ldots, e_m$. If $L=\sum_{j=1}^m r_i e_j \in \g$ and $f \in \OG$, then we can apply $L$ to $f$ by \be \label{lf} L(f)=\sum_{j=1}^m r_j \frac{\partial f}{\partial t_1^{(j)}}(0).\ee  Hence we can identify $\g$ with the set of $R$-derivations of $\OG$ into $R$, where $R$ is considered as a $\OG$-module via evaluation at zero. We denote this set by $\Der_0(\OG, R)$.

The elements $\gamma_{lk}^j$ of Definition \ref{law} of a formal group law $G$ define a Lie algebra structure on $\g$, as follows (see \cite[Chap.II, p.79]{haze}): \begin{equation} \label{lb} [e_l,e_k]= \sum_{j=1}^{m}(\gamma_{lk}^j-\gamma_{kl}^j)e_j.\end{equation} However, $\g$ inherits also a Lie algebra structure by the canonical bijection of $\Der_R(\OG, \OG)$ and $\Der_0(\OG, R)$. One can check that both definitions of the Lie-bracket coincide.

Let $\mathcal{U}(\mathfrak{g})$ denote the universal enveloping algebra of $\g$. The theorem of Poincar\'{e}-Birkhoff-Witt (\cite[Chap.I.2.7, Thm. 1., Cor. 3.]{bour}) says that the underlying set of $\mathcal{U}(\mathfrak{g})$ is the polynomial ring $R[e_1, \ldots, e_m]$. Therefore we denote an arbitrary element of $\mathcal{U}(\mathfrak{g})$ by $\sum c_{\underline{j}}\underline{e}^{\underline{j}}$ with $c_{\underline{j}} \in R$.

\subsection{Hopf algebra structures} \label{sechopf}
The algebras $\OG$ and $\mathcal{U}(\mathfrak{g})$ carry a (complete) Hopf algebra structure. In the following we will describe the maps defining these structures. As a reference for the definition of a Hopf algebra one can take the book of M.E. Sweedler about Hopf algebras, \cite[Chap.I-VI.]{sweedler}, or the book of Ch. Kassel about Quantum groups, \cite[Chap.III]{kassel}. For the definition of a complete Hopf algebra we have to consider complete $R$-modules, i.e. topologized $R$-modules which are complete with respect to a given topology. In our cases this topology will come from a descending filtration $\{F^nM\}$ on the $R$-module $M$. By replacing $R$-modules by complete $R$-modules, and also tensor products by complete tensor products, one can define in the same way as for Hopf algebras a \textit{complete Hopf algebra} over $R$.

\begin{prop} \cite[Chap. VII.36]{haze} \label{hog}
Let $G$ be a formal group law. Then the ring of functions on $G$ carries a complete Hopf algebra structure $(\OG,\mathcal 5, \eta, \mu, \epsilon, s)$. The maps are given by \begin{center} \begin{tabular}{cccccccccc}
$\mathcal 5$:& $\OG \hat{\otimes} \OG$ &$\rightarrow$& $\OG$& & &$\eta:$& $R$ &$\rightarrow$& $\OG$ \\& $f \otimes g$ &$\mapsto$& $f\cdot g$& & & & $1$ &$\mapsto$& $1$ \end{tabular} \end{center}
\begin{center} \begin{tabular}{cccccccccc}
$\mu$:& $\OG$ &$\rightarrow$& $\OG \hat{\otimes} \OG$& &$\epsilon:$& $\OG$ &$\rightarrow$& $R$ \\& $t^{(i)} $ &$\mapsto$& $G^{(i)}(t_1^{(1)} , \ldots,  t_1^{(m)} , t_2^{(1)}, \ldots,  t_2^{(m)})$& & & $f$ &$\mapsto$& $f(0)$ \end{tabular} \end{center}
\begin{center} \begin{tabular}{cccc}
$s$:& $\OG$ &$\rightarrow$& $\OG$\\& $t^{(i)} $ &$\mapsto$& $s(t^{(i)}) $,\end{tabular} \end{center} where the antipode map $s$ is given by Proposition \ref{exs} by the condition that $$G^{(i)}(t^{(1)} , \ldots,  t^{(m)} , s(t^{(1)}), \ldots,  s(t^{(m)}))=0.$$
\end{prop}

\begin{defnprop}
Let $\D$ be the continuous dual of $\OG$, i.e. $\D=\OG^\circ=\varinjlim \Hom_R(\OG/F^n\OG, R)$, where the filtration was given in Lemma \ref{filtog}. Then the complete Hopf algebra structure on $\OG$ yields a Hopf algebra structure on its continuous dual $(\D,\mathcal 5, \eta, \mu, \epsilon, s)$ given by dualizing the structure morphsims of $(\OG,\mathcal 5, \eta, \mu, \epsilon, s)$.
\end{defnprop} 
\begin{proof} See \cite[Chap.VII.36]{haze} and note that $\D$ is in our case actually a Hopf algebra, since we didn't require that the antipode $s$ is a $\D$-module homomorphism. \end{proof}

Note that we have associated to a formal group law $G$ the complete Hopf algebra $\OG$ and the Hopf algebra $\D$. These objects are dual to each other, where one gets from $\OG$ to $\D$ by taking continuous linear duals and from $\D$ to $\OG$ by taking linear duals. This duality extends to the categories formed by these objects and is known as Cartier duality, see for example \cite[Chap.I.2]{die}.

\begin{prop} \cite[Chap.II.14.3]{haze} \label{hug}
Let $\g$ be a Lie algebra and $\Ug$ the universal enveloping algebra of $\g$. Then $\Ug$ carries a Hopf algebra structure $(\Ug,\mathcal 5, \eta, \mu, \epsilon, s)$. The maps are given by:
\begin{center} \begin{tabular}{cccccccccc}
$\mathcal 5$:& $\mathcal{U}(\mathfrak{g}) \otimes \mathcal{U}(\mathfrak{g})$ &$\rightarrow$& $\mathcal{U}(\mathfrak{g})$& & &$\eta:$& $R$ &$\rightarrow$& $\mathcal{U}(\mathfrak{g})$ \\& $x \otimes y$ &$\mapsto$& $x \cdot y$& & & & $1$ &$\mapsto$& $1$ \end{tabular} \end{center}
\begin{center} \begin{tabular}{cccccccccc}
$\mu$:& $\g$ &$\rightarrow$& $\mathcal{U}(\mathfrak{g}) \otimes \mathcal{U}(\mathfrak{g})$& & &$\epsilon:$& $\g$ &$\rightarrow$& $R$ \\& $L$ &$\mapsto$& $L \otimes 1 + 1 \otimes L$& & & & $L$ &$\mapsto$& $0$\end{tabular} \end{center}
\begin{center} \begin{tabular}{cccc}
$s$:& $\g$ &$\rightarrow$& $\mathcal{U}(\mathfrak{g})$\\& $L$ &$\mapsto$& $-L $.\end{tabular} \end{center}
\end{prop}

Let $\g$ be an $m$-dimensional Lie algebra over $R$ which is free as an $R$-module with basis $e_1, \ldots , e_m$ and let $\Ug$ be the universal enveloping algebra of $\g$. Then we denote by $\U*(\g,R):=\Hom_{R}(\Ug,R)$ or by $\U*$ - if $\g$ and $R$ are clear from the context - the $R$-linear dual of $\Ug$. 

Let $d^{j_1}t^{(1)} \cdots d^{j_m}t^{(m)}$ be the dual basis of $e_i^{j_1} \cdots e_m^{j_m}$ with $j_i \in \{1,\ldots n\}$ for all $i \in \{1,\ldots m\}$. Then $\U*$ has a ring structure with underlying set \be \label{dual} R\{\{d\underline{t}\}\}:=\prod_{\underline{j}} R d^{j_1}t^{(1)} \cdots d^{j_m}t^{(m)}\ee and the two binary operations $+$ as usual addition and multiplication $\mal$ given by the comultiplication of $\Ug$: 
\bes \mal: \U* \otimes \U* &\rightarrow& \U* \\ \psi \otimes \varphi &\mapsto& \left[x \mapsto \rho(\psi \otimes \varphi)(\mu(x))\right],\ees where $\rho: \U* \otimes \U* \rightarrow (\Ug \otimes \Ug)^\star$ is the linear injection given by $$\rho(\psi \otimes \varphi)(x \otimes y)=\psi(x) \cdot \varphi(y).$$ 
Analogous to $\OG$ we can define a filtration on $\U*$ by 
\begin{eqnarray} \label{filu}
\mathcal{F}^i\U* =\left\lbrace \varphi \in \U* | \mbox{for all monomials} \ d^{\underline{j}}\underline{t} \ \mbox{is} \ |\underline{j}| \geq i \right\rbrace,\end{eqnarray} for all $i \in \mathbb{N}$, such that $\U*$ is a completed ring with respect to the topology induced by this filtration. And we can identify $\U* \hat{\otimes} \U*$ with $(\Ug \otimes \Ug)^\star$. The underlying set of $(\U*)^{\hat{\otimes}n}$ is given by $R\{\{d\underline{t}_1, \ldots, d\underline{t}_n\}\}$, where we use the multi-index notation $d^{\underline{r}}\underline{t}_i$ - or equivalently $1 \otimes \ldots \otimes 1 \otimes d^{\underline{r}}\underline{t} \otimes 1 \otimes \ldots \otimes 1$ with $d^{\underline{r}}\underline{t}$ at the $i$-th entry - for $d^{r_1}t_i^{(1)} \cdots d^{r_m}t_i^{(m)}$.

\begin{prop}
Let $\g$ be a free Lie algebra, $\Ug$ the universal enveloping algebra of $\g$ and $\U*$ the linear dual of $\Ug$. Then $\U*$ carries a complete Hopf algebra structure $(\U*,\mal, \eta, \mu, \epsilon, s)$. The maps are given by:
\begin{center} \begin{tabular}{cccccccccc}
$\mal$:& $\U* \hat{\otimes} \U*$ &$\rightarrow$& $\U*$& & &$\eta:$& $R$ &$\rightarrow$& $\U*$ \\& $\psi \otimes \varphi$ &$\mapsto$& $\left[x \mapsto \rho(\psi \otimes \varphi)(\mu(x))\right]$& & & & $1$ &$\mapsto$& $\left[x \mapsto \epsilon(x)\right] $ \end{tabular} \end{center}
\begin{center} \begin{tabular}{cccccccccc}
$\mu$:& $\U*$ &$\rightarrow$& $\U* \hat{\otimes} \U*$& & &$\epsilon:$& $\U*$ &$\rightarrow$& $R$ \\& $\varphi$ &$\mapsto$& $\left[x \otimes y \mapsto \varphi(\mathcal 5(x \otimes y))\right]$& & & & $\varphi$ &$\mapsto$& $\varphi(1)$ \end{tabular} \end{center}
\begin{center} \begin{tabular}{cccc}
$s$:& $\U*$ &$\rightarrow$& $\U*$\\& $\varphi$ &$\mapsto$& $\left[x \mapsto \varphi(s(x))\right] $.\end{tabular} \end{center} \end{prop}
\begin{proof} Note first that again since $\mal$ is a continuous map it is enough to define this map on $\varphi \otimes \psi \in \U* \otimes \U*$. After the definition of a Hopf algebra we have to check that $(\U*,\mal, \eta, \mu, \epsilon)$ is a complete bialgebra. For this see \cite[Chap.I-IV]{sweedler} and note that the finiteness condition in Sweedlers book can in our case be replaced by the identification $$(\Ug \otimes \Ug)^\star \cong \U* \hat{\otimes} \U*.$$ Secondly we have to show that $s$ is actually an antipode, i.e. that $$\mal \circ (s \otimes 1) \circ \mu=\eta \circ \epsilon= \mal \circ (1 \otimes s) \circ \mu$$ but this can be easily verified from the antipode condition of $\Ug$. \end{proof}

In the following lemma, we will provide explicit formulas for the multiplication and comultiplication in $\U*$. Especially the explicit formula for the multiplication will play an essential role in the next section.

\begin{lemma} \cite[Chap.V.5-6]{serrelie} \label{expl}  Let $\g$ be an $m$-dimensional Lie algebra over $R$ which is free as an $R$-module with basis $e_1, \ldots , e_m$ and let $\U*$ be the linear dual of the universal enveloping algebra of $\g$. Let $d^{j_1}t^{(1)} \cdots d^{j_m}t^{(m)}$ be the dual basis of $e_i^{j_1} \cdots e_m^{j_m}$ with $j_i \in \{1,\ldots n\}$ for all $i \in \{1,\ldots m\}$, so that an element of $\U*$ is of the form $\sum_{\underline{j}} c_{\underline{j}}d^{j_1}t^{(1)} \cdots d^{j_m}t^{(m)}$ with $c_{\underline{j}} \in R$. Then multiplication and comultiplication in $\U*$ are given by the continuous $R$-linear extension of
\begin{center} \begin{tabular}{ccccccccc}
$\mal$:& $\U* \otimes \U*$ &$\rightarrow$& $\U*$ & \mbox{and} & $\mu$:& $\U*$ &$\rightarrow$& $\U* \hat{\otimes} \U*$ \\& $d^{\underline{r}}\underline{t} \otimes d^{\underline{s}}\underline{t}$ &$\mapsto$& $\binom{\underline{r}+\underline{s}}{\underline{r}} d^{\underline{r}+\underline{s}}\underline{t}$& & & $d^{\underline{r}}\underline{t}$&$\mapsto$& $\sum_{\underline{l}+\underline{k}=\underline{r}} d^{\underline{l}}\underline{t} \otimes d^{\underline{k}}\underline{t}.$\end{tabular} \end{center}
\end{lemma}

\begin{ex} \label{exmal}
Let $m=1$. Then the explicit formulas for multiplication and comultiplication amount to
$$\underbrace{dt \mal \cdots \mal dt}_{n-\mbox{times}}= \underbrace{dt \mal \cdots \mal dt}_{(n-2)-\mbox{times}} \mal 2d^2t = n!d^nt$$ $$\mu(dt)=1 \otimes dt + dt \otimes 1.$$
\end{ex}

\subsection{Cohomology complexes} \label{cc}

\begin{defnprop} \cite[Chap.XVIII.5]{kassel} \label{hc}
Let $(H,\mathcal 5, \eta, \mu, \epsilon,s)$ be a complete Hopf algebra over $R$. Set $T^n(H)=H^{\hat{\otimes}n}$ if $n >0$ and \hbox{$T^0(H)=R$}. We define linear maps $\partial_n^0, \ldots, \partial_n^{n+1}$ from $T^n(H)$ to $T^{n+1}(H)$ by the continuous extension of \begin{align*}
\partial_n^0(x_1 \otimes \cdots \otimes x_n) &= 1 \otimes x_1 \otimes \cdots \otimes x_n, \\
\partial_n^{n+1} (x_1 \otimes \cdots \otimes x_n) &=x_1 \otimes \cdots \otimes x_n \otimes 1 , \\
\partial_n^{i} (x_1 \otimes \cdots \otimes x_n) &=x_1 \otimes \cdots \otimes x_{i-1} \otimes \mu(x_i) \otimes x_{i+1} \otimes \cdots \otimes x_n  ,\end{align*} if $1\leq i \leq n$. If $n=0$, we set $\partial_0^0(1)=\partial_0^1(1)=1$. We have $\partial_{n+1}^j\partial_n^i=\partial^i_{n+1}\partial_n^{j-1}$ for all integers $i,j$ such that $0\leq i < j\leq n+2.$ We define the differential $\partial: T^n(H) \rightarrow T^{n+1}(H)$ by \be \partial=\sum_{i=0}^{n+1}(-1)^{i}\partial_n^i. \label{diff} \ee Then $\partial \circ \partial =0$ and we obtain a cochain complex $(T^\bullet(H), \partial)$ called \textit{cobar complex} of the complete Hopf algebra $H$. \end{defnprop}

In the case of the ring of functions $\OG$ to a formal group law and in the case of the dual of the universal enveloping algebra of $\mathfrak{g}$ we will use the following notation.
\begin{defn} \label{gc}
Let $G$ be a formal group law. An \textit{inhomogeneous n-cochain} of $G$ with coefficients in $R$ is an element of $\OG^{\hat{\otimes} n}$. We will denote the set of inhomogeneous $n$-cochains also by $K^n(G,R)$. The coboundary homomorphisms $\partial^n: K^n(G,R)  \rightarrow  K^{n+1}(G,R)$ of definition \ref{hc} transform into: \begin{align*}  \partial^n(f)(\underline{t}_{1, \ldots, n+1}) &= f(\underline{t}_{2, \ldots, n+1}) \\ &+ \sum_{i=1}^{n} (-1)^{i} f(\underline{t}_1, \ldots, G^{(1)}1(\underline{t}_i, \underline{t}_{i+1}), \ldots , G^{(m)}(\underline{t}_i, \underline{t}_{i+1}),\ldots, \underline{t}_{n+1}) \\ &+ (-1)^{n+1} f(\underline{t}_{1, \ldots, n}).\end{align*}
We obtain a cochain complex $(K^\bullet(G, R), \partial)$ whose cohomology group $H^n(G, R)$ is called \textit{n-th group cohomology of G with coefficients in R.}
\end{defn}

\begin{defnprop} \cite[Chap. I.2]{neukirch} \label{uc}
Let $\mathfrak{g}$ be a Lie algebra over $R$ and let $\mathcal{U}(\mathfrak{g})$ be its universal enveloping algebra. An \textit{inhomogeneous n-cochain} of $\mathcal{U}(\mathfrak{g})$ with coefficients in $R$ is an element of $\Hom_{R}((\mathcal{U}(\mathfrak{g}))^{\otimes n}, R)$. The coboundary homomorphisms \hbox{$\partial_u^n: \Hom_{R}((\mathcal{U}(\mathfrak{g}))^{\otimes n}, R)  \rightarrow  \Hom_{R}((\mathcal{U}(\mathfrak{g}))^{\otimes n+1}, R)$} given by \bes  \partial_u^n(u_1, \ldots u_{n+1}) = f(u_2, \ldots, u_{n+1}) &+& \sum_{i=1}^{n} (-1)^{i} f(u_1, \ldots, u_iu_{i+1},\ldots, u_{n+1})\\ &+& (-1)^{n+1} f(u_1, \ldots u_n) \ees define a cochain complex $(\Hom_{R}(U^\bullet, R),\partial_u)$.\end{defnprop}

\begin{rem} \label{remuc}
Since the set of inhomogeneous $n$-cochains can be identified with $\U*^{\hat{\otimes} n}$, the definition of the coboundary homomorphisms of Definition \ref{uc} is equivalent to the definition of the differential defined by (\ref{diff}) of Definition \ref{hc} for the complete Hopf algebra $\U*$. 
\end{rem}

\begin{defnprop} \label{propu}
A \textit{homogeneous n-cochain} of $\mathcal{U}(\mathfrak{g})$ with coefficients in $R$ is an element of $\Hom_{\mathcal{U}(\mathfrak{g})}((\mathcal{U}(\mathfrak{g}))^{\otimes n+1}, R)$, where $\mathcal{U}(\mathfrak{g})^{\otimes n}$ is considered as an $\mathcal{U}(\mathfrak{g})$-module via the following operation:
$$u.(u_0, \ldots u_{n-1})=(uu_0, \ldots u_{n-1}).$$ The map 
\begin{eqnarray*}
\iota^n: \Hom_{R}(\mathcal{U}(\mathfrak{g})^{\otimes n}, R) &\rightarrow& \Hom_{\mathcal{U}(\mathfrak{g})}(\mathcal{U}(\mathfrak{g})^{\otimes n+1}, R) \\
\varphi &\mapsto& \left[ (u_0,\ldots, u_n) \mapsto \varphi(u_1, \ldots u_n)\right] \\
\left[(u_1, \ldots, u_n) \mapsto \varphi(1,u_1, \ldots u_n)\right] &\reflectbox{$\mapsto$}& \varphi
\end{eqnarray*}is an isomorphism from the set of inhomogeneous to the set of homogeneous $n$-cochains. If we consider the following coboundary homomorphisms $$\partial_{u_h}^n: \Hom_{\mathcal{U}(\mathfrak{g})}((\mathcal{U}(\mathfrak{g}))^{\otimes n+1}, R) \rightarrow \Hom_{\mathcal{U}(\mathfrak{g})}((\mathcal{U}(\mathfrak{g}))^{\otimes n+2}, R)$$ defined by $$ \partial_{u_h}^n= \iota^{n+1} \circ \partial_u^n \circ (\iota^n)^{-1}$$ we obtain a complex $(\Hom_{\mathcal{U}(\mathfrak{g})}(U_h^\bullet, R),\partial_{u_h})$ of \textit{homogeneous n-cochains} and  $\iota^n$ yields an isomorphism $\iota:(\Hom_{R}(U^\bullet, R),\partial_u)  \rightarrow (\Hom_{\mathcal{U}(\mathfrak{g})}(U_h^\bullet, R),\partial_{u_h})$ of complexes.
\end{defnprop}
\begin{proof} It is enough to prove that $\iota$ is an isomorphism, since the remaining statement can be easily deduced from this. We check first that $\iota^n(\varphi)$ is \hbox{$\mathcal{U}(\mathfrak{g})$-invariant}:
$$ (u.\iota^n(\varphi))(u_0, \ldots u_n) = \iota^n(\varphi)(uu_0, \ldots u_n)= \varphi(u_1, \ldots u_n) = \iota^n(\varphi)(u_0, \ldots u_n).$$ 
Secondly we show that $\iota^n$ and $(\iota^n)^{-1}$ are inverse to each other:
\begin{align*}(\iota^n)^{-1} \circ \iota^n(\varphi)(u_1, \ldots, u_n) &= \iota^n(\varphi)(1,u_1, \ldots, u_n) = \varphi(u_1, \ldots, u_n) \\
\iota^n \circ (\iota^n)^{-1}(\varphi)(u_0, \ldots, u_n) &= (\iota^n)^{-1}(\varphi)(u_1, \ldots, u_n) = \varphi(1,u_1, \ldots, u_n) \\ \begin{scriptsize}(\varphi  \ \mbox{is homogeneous}) \end{scriptsize} \hspace{1cm} &= (u_0.\varphi)(1,u_1, \ldots, u_n) = \varphi(u_0,u_1, \ldots, u_n). \end{align*} \end{proof}

For the following definitions let $\g$ be an $m$-dimensional Lie algebra over $R$ which is free as an $R$-module with basis $e_1, \ldots , e_m$. Let $\bigwedge^n\g$ be the $n$-fold exterior product of $\g$ with basis $\{e_{i_1} \wedge \ldots \wedge e_{i_n} \ | \ i_1 <  \ldots < i_n\}, i_j \in \left\lbrace 1,\ldots , m \right\rbrace$. We endow $R$ with the trivial $\g$-action. 

\begin{defnprop}  cite[Chap. IV.3]{knapp}\label{lac} 
The set $\Hom_{R}( \bigwedge^n\g, R)$ is called the set of \textit{n-cochains} of $\g$ with coefficients in $R$ and denoted by $C^n(\g,R)$. Note that the rank of $C^n(\g,R)$ over $R$ is $\binom{m}{n}$. The boundary operators $\partial'^n: C^n \rightarrow C^{n+1}$ are given by the formula 
\bes \partial'^n(\omega)(e_{i_1} \wedge \ldots \wedge e_{i_{n+1}}) = \sum_{1 \leq r < s \leq n+1 } (-1)^{r+s} \omega([e_{i_r},e_{i_s}] \wedge e_{i_1} \wedge \ldots  \wedge e_{i_{n+1}})_{r,s}, \ees
where the notation $([e_{i_r},e_{i_s}] \wedge e_{i_1} \wedge \ldots  \wedge e_{i_{n+1}})_{r,s}$ indicates that the elements $e_{i_r}$ and $e_{i_s}$ are omitted.
We thus obtain, after assuring ourself that $\partial'^2=0$, a complex $(C^\bullet(\g,R),\partial')$ whose cohomology group $H^n(\g,R)$ is called \textit{n-the Lie algebra cohomology of $\g$ with coefficients in $R$}. \end{defnprop}

\begin{defn} \cite[Chap. IV.3]{knapp}\label{kc}
Let $\mathcal{U}(\mathfrak{g})$ be the universal enveloping algebra of $\g$. Set $$V_i(\g)= \mathcal{U}(\mathfrak{g}) \otimes {\bigwedge}^i\g$$ for all $i \in \{ 0, 1, 2, \ldots\}$ with the $\g$-module structure induced by the action on the first factor. The differential $d^{n-1}: V_n(\g) \rightarrow V_{n-1}(\g)$ is given by the formula
\begin{align*} 
d^{n-1}(u \otimes e_{i_1} \wedge \ldots \wedge e_{i_{n}}) =\sum_{1 \leq k < l \leq n} &(-1)^{k+l} (u \otimes [e_{i_{k}},e_{i_{l}}] \wedge e_{i_1} \wedge \ldots \wedge e_{i_{n}})_{k,l} \\ &+ \sum_{j=1}^{n} (-1)^{j+1} (ue_{i_{j}} \otimes e_{i_1} \wedge \ldots \wedge e_{i_{n}})_j,
\end{align*} 
where the notation $(e_{i_1} \wedge \ldots \wedge e_{i_{n}})_{k,l}$, respectively $(e_{i_1} \wedge \ldots \wedge e_{i_{n}})_j$ again indicates that the elements $e_{i_k}$ and $e_{i_l}$, respectively $e_{i_j}$ are omitted. This leads, after assuring ourselves that $d^2=0$, to a complex, called \textit{Koszul complex}. \end{defn}

The following two propositions relate the Koszul complex first to the standard homogeneous complex of $\mathcal{U}(\mathfrak{g})$ and secondly to the Lie algebra complex. 

\begin{prop} \cite[Chap. XIII.7, Theorem 7.1]{cartan} \label{k1prop}
Let $(\V^\bullet, d)$ be the Koszul complex defined above. Then the map $\nu: \Hom_{\mathcal{U}(\mathfrak{g})}(U_h^\bullet, R) \rightarrow \Hom_{\mathcal{U}(\mathfrak{g})}(V(\mathfrak{g})^\bullet, R)$ induced by the anti-sym\-metrisation map 
$$as_n: \bigwedge^n \mathfrak{g} \rightarrow \mathcal{U}^{\otimes n}$$ given by
$$ as_n(e_{i_1}\wedge \ldots \wedge e_{i_n})= \sum_{\alpha \in S_n} \sgn (\alpha) e_{i_{\alpha(1)}} \otimes \cdots \otimes e_{i_{\alpha(n)}},$$
with $i_j \in \left\lbrace 1,\ldots, m \right\rbrace$, is a quasi-isomorphism of complexes.
\end{prop}

\begin{prop} \cite[Chap. IV.3-6]{knapp} \label{k2prop}
Let $(\V^\bullet, d)$ be the Koszul complex defined above. Then the map $\kappa: \Hom_{\mathcal{U}(\mathfrak{g})}(V(\mathfrak{g})^\bullet, R) \rightarrow C^\bullet(\g,R)$ given by \bes \kappa^n: \Hom_{\mathcal{U}(\mathfrak{g})}(V_n(\mathfrak{g}), R) &\rightarrow&  \Hom_{R}(\bigwedge^n\g,R) \\ f &\mapsto& [ (e_{i_1} \wedge \cdots \wedge e_{i_n}) \mapsto f(1 \otimes e_{i_1} \wedge \ldots \wedge e_{i_n})] \ees is an isomorphism of complexes.
\end{prop}

\section{Isomorphism of complete Hopf algebras} \label{chapisobialg} \label{sechom}
Let $R$ be an integral domain of characteristic zero and let $G$ be an $m$-dimensional formal group law and $\OG$ the ring of functions to $G$ (see Definition \ref{law}). Let $\g$ be the associated free $m$-dimensional Lie algebra over $R$ to $G$ which is free as an $R$-module with basis $e_1, \ldots , e_m$ (see description around (\ref{lb})) and let $\U*$ be the linear dual of the universal enveloping algebra of $\g$. This section describes the desired map from $\OG$ to $\U*$ as a composition of maps $\OG \rightarrow \D^\star \rightarrow \U*$, where $\D^\star$ is the linear dual of $\D$ and the latter map is defined by dualizing the map $\mathcal{U}(\mathfrak{g}) \rightarrow \D$. 

\begin{prop} \label{als} Let $G$ be an $m$-dimensional formal group law, $\OG$ the ring of functions to $G$, $\g$ the associated free $m$-dimensional Lie algebra over $R$ to $G$ which is free as an $R$-module with basis $e_1, \ldots , e_m$. Let $\Ug$ be the universal enveloping algebra of $\g$ and $\U*$ the linear dual of $\Ug$. There are natural homomorphisms of (complete) Hopf algebras $$ \al: \mathcal{U}(\mathfrak{g}) \rightarrow \D \ \mbox{and} \ \als : \OG \rightarrow \U*$$ induced by the pairing \begin{eqnarray*}
\g \otimes \OG &\rightarrow& R \\
L \otimes f &\mapsto& L(f),
\end{eqnarray*} 
where $L(f)$ was defined by $L(f)= \sum_{j=1}^m r_j \frac{\partial f}{\partial t_1^{(j)}}(0)$ if $L=\sum_{j=1}^m r_i e_j \in \g$, see Equation (\ref{lf}) in Section \ref{sec1}. 
\end{prop}
\begin{proof}
The existence of the map $\mathcal{U}(\mathfrak{g}) \rightarrow \D$ can be found in \cite[Chap. VII.37.4]{haze} or \cite[Chap.V.6]{serrelie}. The proof that the dual of this map is a morphism of complete Hopf algebras can be deduced from \cite[Chap. VII.37.4]{haze} or \cite[Chap. V.6, THm.2]{serrelie}.
\end{proof}

\begin{notation} \label{notphi}  We will denote the image of $e_i$ under the map 
\begin{eqnarray*}
\g &\xrightarrow{\gamma}& \D \\
L &\mapsto& [f \mapsto L(f)]
\end{eqnarray*} 
by $\phi^{(i)}$, to that we have $\phi^{(i)}(f)= e_i(f)=\frac{\partial f}{\partial t_1^{(i)}}(0) $. The image of $e_i^k$ of $\mathcal{U}(\mathfrak{g})$ under the map $\al$ is therefore denoted by $(\phi^{(i)})^k$, where the multiplication is given by the Hopf algebra structure on $\D$. For $(\phi^{(i)})^0$ one has that $(\phi^{(i)})^0(f)=f(0)$.
\end{notation}

\begin{defn} \label{mO}
Let $Q(R)$ be the quotient field of $R$. To each $m$-dimensional formal group law $G$ we associate an extension of the ring of functions $\OG$ called the \textit{modified ring of functions} $\tOG$ which is defined by $$\tOG:=\left\lbrace  \sum_{\underline{j}} b_{\underline{j}} \underline{t}^{\underline{j}}  \in Q(R)[[\underline{t}]] \ | \ \underline{j}! b_{\underline{j}} \in R \right\rbrace.$$
\end{defn}

\begin{rem}
Note that the modified ring $\tOG$ is actually a ring and that $\OG \subset \tOG$. \end{rem}

\begin{lemma} The modified ring of functions $\tOG$ is a complete ring with respect to the topology induced by the following descending filtration  \begin{eqnarray} \label{filo}
F^i\tOG=\left\lbrace f \in \tOG | \mbox{all monomials of} \ f \ \mbox{have a total degree} \geq i \right\rbrace,\end{eqnarray} for all $i \in \mathbb{N}$. \end{lemma}

\begin{prop} \label{togha}
Let $G$ be an $m$-dimensional formal group law and $\tOG$ the associated modified ring of Definition \ref{mO}. Then $\tOG$ carries a complete Hopf algebra structure $(\tOG,\mathcal 5, \eta, \mu, \epsilon, s)$.
\end{prop}
\begin{proof}
We claim that the morphisms $\mathcal 5, \eta, \mu, \epsilon$ and $s$ of the complete Hopf algebra structure of $\OG$ of Proposition \ref{hog} can be taken to get a complete Hopf algebra structure on $\tOG$. To see this we have to verify the divisibility conditions, i.e. we have to check that if $f=\sum_{\underline{j}} b_{\underline{j}} \underline{t}^{\underline{j}}$ is an element of $\tOG$, then 
\begin{itemize}
\item[(i)] $\epsilon(f) \in R$
\item[(ii)] $\mu(f) \in \tOG \hat{\otimes} \tOG$
\item[(iii)] $s(f) \in \tOG$, \end{itemize} which can be deduced from the definition of the modified ring of functions $\tOG$, the definition of the fromal group law $G$ and the definition of the anipode $s$ together with Remark \ref{exs}.
\end{proof}

\begin{prop} \label{tals}
The map $\als$ of Proposition \ref{als} extends to a homomorphism $\tals: \tOG \rightarrow \U*$ of complete Hopf algebras.
\end{prop}
\begin{proof} Consider the pairing \begin{eqnarray*}
\g \otimes \OG &\rightarrow& R \\
L \otimes f &\mapsto& L(f)
\end{eqnarray*} which induced the homomorphism $\als: \OG \rightarrow \U*$ of complete Hopf algebras of Proposition \ref{als}. This pairing can be extended to a pairing \begin{eqnarray*}
\g \otimes \tOG &\rightarrow& R \\
L \otimes f &\mapsto& L(f)
\end{eqnarray*} since $L(f) \in R$ even if $f \in \tOG$, because the linear terms of $f$ still lie in $R$. With these modifications, the proof of Proposition \ref{als} still goes through.
\end{proof}

\begin{prop}  \label{gen} Let $G$ be an $m$-dimensional formal group law over $R$. Let $\tOG$ be the modified ring of functions (Definition \ref{mO}). Then the map \hbox{$\tals : \tOG \rightarrow \U*$}, defined by Proposition \ref{tals}, is an isomorphism of complete Hopf algebras.
\end{prop}

\begin{rem}
M. Hazewinkel shows in an analogous way in \cite[Chap. VII.37.4]{haze} that the map $\al: \Ug \rightarrow \D$ is an algebra isomorphism if $\OG$ and $\U*$ are $\mathbb{Q}$-algebras and that $\al$ respects the comultiplication and counits. 
\end{rem}

\begin{proof}[Proof of Proposition \ref{gen}]
Consider the filtrations on $\tOG$ and $\U*$ given in (\ref{filo}) and (\ref{filu}) by
\begin{eqnarray*}
F^i\tOG=\left\lbrace f \in \tOG | \mbox{all monomials of} \ f \ \mbox{have a total degree} \geq i \right\rbrace\end{eqnarray*}
and
\begin{eqnarray*} 
\mathcal{F}^i\U*=\left\lbrace \varphi \in \U* | \mbox{for all monomials} \ d^{\underline{j}}\underline{t} \ \mbox{is} \ |\underline{j}| \geq i \right\rbrace\end{eqnarray*}
for all $i \in \mathbb{N}$.
Let $\underline{t}^{\underline{j}}$ be a basis element of $F^i\tOG$, i.e. $|\underline{j}| \geq i$. We claim that $$\tals(\underline{t}^{\underline{j}}) \equiv \underline{j}! d^{\underline{j}}\underline{t} \ \mod  \mathcal{F}^{i+1}\U*.$$
To prove this we will first consider $\tals(t^{(l)})(x)$ with $x=\sum a_{\underline{k}} \underline{e}^{\underline{k}} \in \mathcal{U}(\mathfrak{g})$.
\begin{align*}
\tals(t^{(l)})(x)&= \al(\sum a_{\underline{k}} \underline{e}^{\underline{k}})(t^{(l)}) \\
&= \sum a_{\underline{k}} \al(\underline{e}^{\underline{k}})(t^{(l)}) \\
&= \sum_{|k| \geq 1} a_{\underline{k}} \al(\underline{e}^{\underline{k}})(t^{(l)}) \hspace{1cm} (\mbox{since} \ (\phi^{(k)})^0(t^{(l)})=0) \\
&= a_{\textbf{e(1)}} \phi^{(1)}(t^{(l)}) + \ldots + a_{\textbf{e(m)}} \phi^{(m)}(t^{(l)}) + \sum_{|k| \geq 2} a_{\underline{k}} \al(\underline{e}^{\underline{k}})(t^{(l)})\\
&= a_{\textbf{e(l)}} + \sum_{|k| \geq 2} a_{\underline{k}} \al(\underline{e}^{\underline{k}})(t^{(l)})  \hspace{1cm} (\mbox{since} \ \phi^{(k)}(t^{(l)})=0, \ \mbox{if} \ k\neq l).
\end{align*}
This shows that $\tals(t^{(l)}) \equiv dt^{(l)} \mod (\mathcal{F}^2\U*)$. If we look at the monomial $\underline{t}^{\underline{j}} \in F^i\tOG$ we get that $\tals(\underline{t}^{\underline{j}})$ is of the form  \begin{align*}(dt^{(1)} + \mathcal{F}^2\U*)^{j_1} \mal \cdots \mal (dt^{(m)} + \mathcal{F}^2\U*)^{j_m}, \end{align*} and due to the explicit formula for $\mal$ in Lemma \ref{expl} this means that
\begin{align*} \tals(\underline{t}^{\underline{j}}) & \equiv j_1!d^{j_1}t^{(1)} \cdots j_m!d^{j_m}t^{(m)}  \mod (\mathcal{F}^{i+1}\U*)  \\
&\equiv \underline{j}! d^{\underline{j}}\underline{t} \ \mod  (\mathcal{F}^{i+1}\U*). 
\end{align*}
Since $\tals$ is $R$-linear and continuous it follows that $\tals$ induces isomorphisms $$F^i\tOG/F^{i+1}\tOG \rightarrow \mathcal{F}^i\U*/\mathcal{F}^{i+1}\U*$$ for all  $i \in \mathbb{N}$ and hence, since $\tOG$ and $\U*$ are both complete with respect to these filtrations, that $\tals:\tOG \rightarrow \U*$ is an isomorphism of complete Hopf algebras.
\end{proof}

\section{Quasi-isomorphism of complexes}
Let $R$ be an integral domain of characteristic zero. The main purpose of this section is to show that the isomorphism $$\tals: \tOG \rightarrow \U*$$ defined in Proposition \ref{tals} extends to a quasi-isomorphism $\tphi$ of the corresponding complexes. The underlying morphism $\als: \OG \rightarrow \U*$ was already established in a paper of Huber and Kings concerning a $p$-adic analogue of the Borel regulator and the Bloch-Kato exponential map, see \cite{HK}. They showed that one can directly define a map from locally analytic group cohomology to Lie algebra cohomology by differenting cochains, and that the resulting map is Lazard's comparison isomorphism (\cite[Proposition 4.2.4]{HK}). We will give an explicit description of this quasi-isomorphism and we will see that this description coincides with the one of Huber and Kings, \cite[Definition 1.4.1]{HK}, and hence with Lazard's map.

\begin{defn} \label{mgc}
Let $G$ be a formal group law. Let $\tOG$ be the associated modified ring of functions which inherits by Proposition \ref{togha} a complete Hopf algebra structure. We denote by ${\tK}^n(G,R)$ the $n$-fold complete tensor product of $\tOG$, i.e. $${\tK}^n(G,R)={\tOG}^{\hat{\otimes}n}$$ and hence by $({\tK}^\bullet(G,R), \partial)$ the cobar complex given by Proposition \ref{hc}. The corresponding cohomology group, denoted by $H^n(\tilde{G}, R)$, is called \textit{n-th modified group cohomology of G with coefficients in R}.
\end{defn}

\begin{thm} \label{phi}
Let $R$ be an integral domain of characteristic zero and let $G$ be a formal group law over $R$. Let $(K^\bullet(G,R),\partial)$, $(\tK^\bullet(G,R),\partial)$ and $(C^\bullet(G,R),\partial')$ be the complexes defined in \ref{gc}, \ref{mgc} and \ref{lac}. Then 
\hbox{$\tals: \tOG \rightarrow \U*$}, defined in Proposition \ref{tals}, extends to a quasi-isomorphism $$ \tphi: ({\tK}^\bullet(G,R),\partial) \rightarrow (C^\bullet(\g,R),\partial')$$ given by the following composition of maps \begin{align*} ({\tK}^\bullet(G,R), \partial) \xrightarrow{\tals} (\Hom_{R}(&U^\bullet, R),\partial_u) \xrightarrow[\ref{propu}]{\iota} (\Hom_{\mathcal{U}(\mathfrak{g})}(U_h^\bullet, R),\partial_{u_h}) \\ &\xrightarrow[\ref{k1prop}]{\nu}  (\Hom_{\mathcal{U}(\mathfrak{g})}(V(\mathfrak{g})^\bullet, R), d) \xrightarrow[\ref{k2prop}]{\kappa}  (C^\bullet(\g,R),\partial'). \end{align*} 
In particular, the morphism $\als: \OG \rightarrow \U*$ extends to a morphism $$\phi: (K^\bullet(G,R),\partial) \rightarrow (C^\bullet(\g,R),\partial').$$ 
\end{thm}

\begin{proof}
The proof is essentially the conjunction of all our previous results. In particular, we will use the statements of Proposition \ref{gen} and \ref{propu} and of Proposition \ref{k1prop} and \ref{k2prop} concerning the Koszul complex.

Consider the following composition of maps of the theorem 
\begin{align} \label{comp} ({\tK}^\bullet(G,R), \partial) \xrightarrow{\tals} (\Hom_{R}(&U^\bullet, R),\partial_u) \xrightarrow[\ref{propu}]{\iota} (\Hom_{\mathcal{U}(\mathfrak{g})}(U_h^\bullet, R),\partial_{u_h}) \\ &\xrightarrow[\ref{k1prop}]{\nu}  (\Hom_{\mathcal{U}(\mathfrak{g})}(V(\mathfrak{g})^\bullet, R), d) \xrightarrow[\ref{k2prop}]{\kappa}  (C^\bullet(\g,R),\partial') \nonumber .\end{align} 
We will recall from Propositions \ref{propu}, \ref{k1prop} and \ref{k2prop} that the maps $\iota, \nu$ and $\kappa$ of (\ref{comp}) are (quasi-)isomorphisms, and we will show that the first map is induced by the map $\tals$ of Proposition \ref{tals} and therefore is even an isomorphism of complexes.

As a first step, recall from Proposition \ref{gen} that there exists an isomorphism $$\tals: \tOG \rightarrow \U*$$ of complete Hopf algebras. This isomorphism $\tals$ extends naturally to a morphism of the complexes $$({\tK}^\bullet(G,R),\partial) \xrightarrow{\tals} (\Hom_{R}(U^\bullet, R),\partial_u).$$ To see this, note first that we can identify $(\U*)^{\hat{\otimes}n}$ with $\Hom_{R}(\mathcal{U}(\mathfrak{g})^{\otimes n},R)$, compare Remark \ref{remuc}, and secondly that the differentials of the complexes $({\tK}^\bullet(G,R),\partial)$ and $(\Hom_{R}(U^\bullet, R),\partial_u)$ are given by the comultiplication of $\tOG$ and $\U*$. The latter means that $\tals$ commutes with these differentials, since $\als$ is in particular a coalgebra morphism.

Secondly recall Proposition \ref{propu}, which stated that there exists an isomorphism between the inhomogeneous and homogeneous complex of $\mathcal{U}(\mathfrak{g})$, hence we get  an isomorphism of complexes:
\bes(\Hom_{R}(U^\bullet, R),\partial_u) \xrightarrow{\iota} (\Hom_{\mathcal{U}(\mathfrak{g})}(U_h^\bullet, R),\partial_{u_h}). \ees 
Finally we can conclude the proof of the theorem by recalling both Propositions \ref{k1prop} and \ref{k2prop} concerning the Koszul complex, whose combined statement is that the map 
$$(\Hom_{\mathcal{U}(\mathfrak{g})}(U_h^\bullet, R),\partial_{u_h}) \rightarrow  (C^\bullet(\g,R),\partial') $$ is a quasi-isomorphism. 
\end{proof}

\begin{cor} \label{cor1}Let $R$ be an integral domain of characteristic zero and let $G$ be a formal group law over $R$. Let $\g$ be the associated Lie algebra to $G$. Then there exists an isomorphism $$\tilde{\Phi}: H^n(\tilde{G}, R) \rightarrow H^n(\g,R)$$ between the modified group cohomolgy of $G$ with coefficients in $R$ and Lie algebra cohomology of $\g$ with coefficients in $R$ given by $f \mapsto \tphi(f)$, with $\tphi(f)$ defined in Theorem \ref{phi}.
\end{cor} 

\begin{cor} \label{phiexpl} Let $R$ be an integral domain of characteristic zero and let $G$ be a formal group law over $R$. Let $f \in  K^n(G,R)$ be given by $f=f_1 \otimes \cdots \otimes f_n$ with $f_i=\sum_{\underline{j}} b_{\underline{j}}^i \underline{t}_i^{\underline{j}}$. Then the map $\phi^n: K^n(G,R) \rightarrow C^n(\g,R)$ of Theorem \ref{phi} can be described by the continuous extension of 
\begin{align*} f_1 \otimes \cdots \otimes f_n \mapsto df_1 \wedge \cdots \wedge df_n \end{align*}
for $n \geq 1$ and by the identity for $n=0$. \end{cor}

\section{Locally analytic group cohomology}
This section will provide a comparison isomorphism of locally analytic group cohomology with Lie algebra cohomology for $K$-Lie groups, where $K$ is a finite extension of $\Q_p$. The main idea of the proof is its reduction to standard groups associated to formal group laws. Before stating the Comparison Theorem for standard groups \ref{thmstd} we will give some definitions and fix notation in Section \ref{intrsec}. Section \ref{seclimit} covers the so called limit morphism, a preliminary step in the proof of the Comparison Theorem. The final steps to the proof of the Comparison Theorem \ref{thmstd} will be given in Section \ref{secmthm}.

\begin{quote} Throughout this section let $K$ be a finite extension of $\Q_p$.\end{quote} 

\subsection{K-Lie groups and standard groups} \label{intrsec}
Let $|\phantom{a}|: K \rightarrow \mathbb{R}_+$ be the non-archimedean absolute value on $K$ which extends the $p$-adic absolute value on $\Q_p$ and let $v_p(x)$ denote the corresponding valuation on $K$, normalized by $v_p(p)=1$, which satisfies $|x|=p^{-v_p(x)}.$

We denote by $R$ the \textit{valuation ring} $$R=\{x \in K : \ |x| \leq 1\}=\{x \in K : \ v_p(x) \geq 0\}$$ 
and by $\mathfrak{m}$ the \textit{maximal ideal} $$\mathfrak{m}=\{x \in K : \ |x| < 1\}=\{x \in K : \ v_p(x) > 0\}$$ of $R$. 

\begin{defn} \label{locanfct} \begin{itemize} \item[(i)] Let $U \subset K^n$ be open and let $f: U\rightarrow K$ be a function. Then $f$ is called \textit{locally analytic in $U$} if for each $x \in U$ there is a ball $B_r(x):=\{y \in U \mid |y-x|<r\} \subset U$ and a formal power series $F$ such that $F$ converges in $B_r(x)$ and for $h \in B_r(x)$, $$f(h)=F(h-x),$$ compare \cite[Chap.III, 1.3.2]{lazard}.
\item[(ii)] Let $U \subset K^n$ be open and let $f=(f_1,\ldots, f_n): U\rightarrow K^n$. Then $f$ is called \textit{locally analytic in U} if $f_i$ is locally analytic for $1\leq i \leq n$.
\item[(iii)] Let $M$ be a topological space. A \textit{chart for $M$} is a triple $(U, \varphi, K^n)$ consisting of an open subset $U \subset M$ and a map $\varphi: U \rightarrow K^n$ such that $\varphi(U)$ is open in $K^n$ and $\varphi: U \xrightarrow{\simeq} \varphi(U)$ is a homeomorphism.
\item[(iv)] A \textit{locally analytic manifold over K} is a topological space $M$ equipped with a maximal atlas, where the atlas is a set of charts for $M$ any two of which are compatible, i.e. has locally analytic transition maps, and which cover $M$. 
\item[(v)] A \textit{$K$-Lie group} or a \textit{$K$-analytic group} $\Gc$ is a locally analytic manifold over $K$ which also carries the structure of a group such that \begin{itemize}
\item[(1)] the function $(x,y) \mapsto xy$ of $G \times G$ into $G$ is locally analytic and
\item[(2)] the function $x \mapsto x^{-1}$ of $G $ into $G$ is locally analytic. 
\end{itemize} \end{itemize} \end{defn}
We refer to \cite{serrelie}, \cite{bour} or \cite{dix} for the background on $K$-Lie groups.

\begin{defnprop} \cite[Chap. IV.8]{serrelie} \label{g1} Let $G$ be an $m$-dimensional formal group law over the valuation ring $R$ of $K$ as in Definition \ref{law}. For $h \in \mathbb{R}$ we set $$\Gc(h):=\{\underline{z} \in R^m \mid |\underline{z}| < p^{-h} \}.$$ We define a multiplication on $\Gc(h)$ by: \begin{eqnarray} \label{mu}
\underline{z_1} \cdot \underline{z_2}=G(\underline{z_1},\underline{z_2}).
\end{eqnarray} Then $\Gc(h)$ is a $K$-Lie group. \end{defnprop}

\begin{defn} \label{std} A $K$-Lie group constructed as in Proposition \ref{g1} will be called an \textit{$m$-standard group of level h} if $h >0$ and just an \textit{$m$-standard group} if $h=0$. \end{defn}

\begin{defn} Let $\Gc$ be a $K$-Lie group. We denote by  $L(\Gc)=T_e(\Gc)$ the canonical Lie algebra of $\Gc$, see \cite[Chap. 4, §16.3]{tu}. \end{defn}

\begin{prop} Let $G$ be a formal group law over $R$ and let $\Gc(h)$ be the $m$-standard group of level $h$ to $G$. Let $L(\Gc(h))$, respectively $\g$ be the corresponding Lie algebras. Then $$L(\Gc(h)) \cong \g \otimes_R K.$$ \end{prop}
\begin{proof} (Compare \cite[Prop. 17.3]{schn}.) One can choose the local coordinates $(t^{(1)}, \ldots, t^{(m)})$ in a neighborhood of the identity $e$ (with coordinates $(0,\ldots,0)$), such that we get a natural basis $\frac{\partial}{\partial t^{(1)}}\big|_{\mbox{\begin{tiny}$(0)$\end{tiny}}},\ldots,\frac{\partial}{\partial t^{(m)}}\big|_{\mbox{\begin{tiny}$(0)$\end{tiny}}}$ of the tangent space $T_e(\Gc(h)).$ Let the formal group law $G$ be given by $$G^{(j)}(\textbf{X},\textbf{Y}) = X_j+ Y_j+ \sum_{l,k=1}^{m} \gamma_{lk}^j X_lY_k +O(d \geq 3)$$ for all $j \in \{1,\ldots m\}$. Then the structure coefficients, i.e. those elements $c_{ij}^k \in R$ such that $$[\frac{\partial}{\partial t^{(i)}}\big|_{\mbox{\begin{tiny}$(0)$\end{tiny}}}, \frac{\partial}{\partial t^{(j)}}\big|_{\mbox{\begin{tiny}$(0)$\end{tiny}}}]= \sum c_{ij}^k \frac{\partial}{\partial t^{(k)}}\big|_{\mbox{\begin{tiny}$(0)$\end{tiny}}},$$ of $L(\Gc(h))$ are given by $\sum_{j=1}^{m}(\gamma_{lk}^j-\gamma_{kl}^j)$ since multiplication on $\Gc(h)$ is defined by the formal group law, see (\ref{mu}). Hence the definition of the Lie bracket in $L(\Gc(h))$ coincides with the definition of the Lie bracket in $\g$, see (\ref{lb}). \end{proof}

\begin{defn} Let $\Gc$ be a $K$-Lie group. We denote by $\mathcal{O}^{la}(\Gc, K)$ locally analytic functions on $\Gc$, i.e. those that can be locally written as a converging power series with coefficients in $K$, see Definition \ref{locanfct}. The cobar complex $\mathcal{O}^{la}(\Gc^n, K)_{n \geq 0}$, where we identified $\mathcal{O}^{la}(\Gc^n, K)$ with ${\mathcal{O}^{la}(\Gc, K)}^{\hat{\otimes}n}$, with the usual differential as in Definition \ref{gc} leads to \textit{locally analytic group cohomology} whose $n$-th cohomology group is denoted by $H^n_{la}(\Gc, K)$. \end{defn}

Let $G$ be a formal group law over $R$. Then one can consider the same formal group law over $K$, denote by $G_K$, with Lie algebra $\g_K$. Let $\Gc(h)$ be the $m$-standard group of level $h$ associated to $G$. Then we obtain by assigning to each locally analytic function $f \in \mathcal{O}^{la}(\Gc(h), K)$ its local power series representation around $e$ a morphism of complete Hopf algebras $$\mathcal{O}^{la}(\Gc(h), K) \rightarrow \OGK$$ and hence a map $$\Phi_e:H^n_{la}(\Gc(h), K) \rightarrow H^n(G_K, K).$$ 

Recall that in the case of a formal group law over the field $K$ the modified ring of functions $\tOGK$ coincides with the ring of functions $\OGK$. According to Corollary \ref{cor1} and since $$\Hom_{R}({\bigwedge}^n\g, R) \otimes_R K \cong \Hom_{K}({\bigwedge}^n\g_K, K)$$ we get an isomorphism $$\Phi_K: H^n(G_K, K) \rightarrow H^n(\g, K).$$ This isomorphism is, as we have seen in Proposition \ref{phiexpl}, given by the continuous extension of differentiating cochains. 

\begin{thm}[Comparison Theorem for standard groups] \label{thmstd} Let $G$ be a formal group law over $R$ and let $\Gc(h)$ be the $m$-standard group of level $h$ to $G$ with Lie algebra $\g \otimes_R K$. Then the map \begin{align*} \Phi_s: H^n_{la}(\Gc(h), K) \rightarrow H^n(\g, K) \end{align*} given by the continuous extension of \begin{align*} f_1 \otimes \cdots \otimes f_n \mapsto df_1\wedge \cdots \wedge df_n, \end{align*}
for $n\geq 1$ and by the identity for $n=0$ is an isomorphism for all $h~>~h_0~=~\frac{1}{p-1}$. \end{thm}

Note that since the elements of the form $f_1 \otimes \cdots \otimes f_n$ form a basis of the dense subset $\mathcal{O}(G_K)^{\otimes n} \subset \mathcal{O}(G_K)^{\hat{\otimes}n}$ and since $\Phi_s$ is given by the composition $\Phi_K \circ \Phi_e$, we use the suggestive notation $f_1 \otimes \cdots \otimes f_n$ for an element of $\mathcal{O}^{la}(\Gc(h), K)$. 

\begin{rem} Huber and Kings showed in \cite{HK} that one can directly define a map from locally analytic group cohomology to Lie algebra cohomology by differenting cochains, as in Theorem \ref{thmstd}, and that in the case of smooth algebraic group schemes $H$ over $\Z_p$ with formal group $\mathcal{H} \subset H(\Z_p)$ the resulting map $$ \Phi: H^n_{la}(\mathcal{H}, \Q_p) \rightarrow H^n(\mathfrak{h}, \Q_p)$$ coincides with Lazard's comparison isomorphism (\cite[Theorem 4.7.1]{HK}). In their joint work with N. Naumann in \cite{HKN} they extended the comparison isomorphism for $K$-Lie groups attached to smooth group schemes with connected generic fibre over the integers of $K$ (\cite[Theorem 4.3.1]{HKN}). \end{rem}

\begin{rem} \label{sketch}Let us sketch the argument of the proof. In the first step, we are going to show that if we restrict to $\OGe$, the ring of germs of locally analytic functions on $\Gc(0)$ in $e$, we can show that the limit morphism $$\Phi_\infty: H^n(\mathcal{O}^{la}(\mathcal{G}(0)^\bullet)_e, K) \rightarrow H^n(\mathfrak{g}, K)$$ is an isomorphism, see Lemma \ref{limit}. Then injectivity of $\Phi_s$ follows from a spectral sequence argument, see Corollary \ref{corinj1} and \ref{corinj2}. The proof of this part will be analogous to the proof of Theorem 4.3.1 in \cite{HKN}, however independent of the work of Lazard \cite{lazard}. For surjectivity we will use the statement of Theorem \ref{phi}. \end{rem}

The Main Theorem of the introduction - which stated that if $K$ is a finite extension of $\Q_p$ and if $\Gc$ is a $K$-Lie group, then there exists an open subgroup $\mathcal{U}$ of $\Gc$ such that the Lazard morphism $$\Phi_L: H^n_{la}(\mathcal{U}, K) \rightarrow H^n(\mbox{Lie}(\mathcal{U}), K)$$ induced by differentiating cochains is an isomorphism - can now be deduced from the following Lemma \ref{subgr}. 

\begin{lemma} \cite[Chap. IV.8]{serrelie} \label{subgr} Any $K$-Lie group contains an open subgroup which is an $m$-standard group. \end{lemma}

\subsection{Limit morphism} \label{seclimit}
This section contains the first step in the proof of Theorem \ref{thmstd}. We will prove that the limit morphism $\Phi_\infty~:~H^n(\mathcal{O}^{la}(\mathcal{G}(0)^\bullet)_e,K)~\rightarrow~H^n(\mathfrak{g},K)$, mentioned in Remark \ref{sketch}, is an isomorphism. To do this we need some further definitions and a comparison of germs.

\begin{prop} \label{rel} Let $G$ be a formal group law over $R$ and let $\Gc(h)$ be the $m$-standard group of level $h$ to $G$. Then there exists an isomorphic $m$-dimensional formal group law $G_h$ over $R$ such that the associated $m$-standard group $\Gc_h(0)$ is equal to $\Gc(h)$. \end{prop}
\begin{proof} Let the formal group law $G$ be given by $$G^{(j)}(\textbf{X},\textbf{Y}) = X_j+ Y_j+ \sum_{l,k=1}^{m} \gamma_{lk}^j X_lY_k +O(d \geq 3)$$ for all $j \in \{1,\ldots, m\}$  with $\textbf{X}=(X_1, \ldots, X_m)$, $\textbf{Y}=(Y_1,\ldots, Y_m)$ and with $\OG=R[[\underline{t}]]$. Consider the $m$-dimensional formal group law defined by 
$$G_h^{(j)}(\textbf{X},\textbf{Y}) = p^{-h} G^{(j)}(p^h\textbf{X},p^h\textbf{Y}) \ \mbox{for all} \ j \in \{1,\ldots, m\}$$ with $p^h\textbf{X}=(p^hX_1, \ldots, p^hX_m)$ and $p^h\textbf{Y}=(p^hY_1,\ldots, p^hY_m)$. Then the ring of functions to $G_h$ is given by $R[[p^{-h}\underline{t}]]$ where $p^{-h}\underline{t}$ is the short notation for $p^{-h}t^{(1)}, \ldots,  p^{-h}t^{(m)}$. Now the $m$-standard group $$\Gch(0) = \{\underline{z} \in R^m \mid |p^{-h}\underline{z}| < 1 \}$$ can be rewritten as following 
\begin{align*}
\Gch(0) = \{\underline{z} \in R^m \mid p^{h}|\underline{z}| < 1 \} = \{\underline{z} \in R^m \mid |\underline{z}| < p^{-h} \} = \Gc(h). \end{align*}
The homomorphism from $G$ to $G_h$ is given by the $m$-tuple $$\alpha(\textbf{X})=(\alpha_1(\textbf{X}), \ldots ,\alpha_m(\textbf{X})) \ \mbox{with} \ \alpha_j(\textbf{X})=p^{-h}X_j$$ and the homomorphism from $G_h$ to $G$ is given by the $m$-tuple $$\beta(\textbf{X})=(\beta_1(\textbf{X}), \ldots ,\beta_m(\textbf{X})) \ \mbox{with} \ \beta_j(\textbf{X})=p^hX_j,$$ compare Definition \ref{hom}. Since they satisfy the condition $\alpha(\beta(\textbf{X}))= \beta(\alpha(\textbf{X}))$ the formal group laws $G$ and $G_h$ are isomorphic. 
\end{proof}

\begin{ex} Let $G_m(X,Y)=X+Y+XY$ be the multiplicative formal group law. Then the formal group law $G_h$ such that $\Gch(0)=\Gc_m(h)$ is given by $$G_h(X,Y)=p^{-h}(p^hX+p^hY+p^{2h}XY)= X+Y+p^hXY.$$ \end{ex}

\begin{lemma} \label{open} Let $G$ be an $m$-dimensional formal group law over $R$. Then the associated $m$-standard groups $\Gc(h)$ of level $h$, $h \in \mathbb{N}\setminus\{0\}$, are open and normal subgroups of $\Gc(0)$ of finite index and they form a neighbourhood basis of $e$ in $\Gc(0)$. \end{lemma}
\begin{proof} The $m$-standard groups $\Gc(h)$ of level $h$ are, by their definition in \ref{g1}, obviously open and closed subgroups of $\Gc(0)$ and form a neighbourhood basis of $e$ in $\Gc(0)$. Since $R^m$ is compact, see \cite[Prop. 5.4.5vi]{gou}, $\Gc(0)$ is compact and the open and closed subgroups $\Gc(h)$ of $\Gc(0)$ are of finite index. For the property that these subgroups are normal we have to show that  $$G(\textbf{x},G(\textbf{y},s(\textbf{x}))) \equiv 0 \ (\mbox{mod} \ p^h),$$ for $\textbf{x} \in \Gc(0)$ and $\textbf{y} \in \Gc(h)$, i.e. $ \textbf{y} \equiv 0 \ (\mbox{mod} \ p^h)$, where $s(\textbf{x})$ was defined by the condition that $G(\textbf{x},s(\textbf{x}))=0$, see Proposition \ref{exs}. However since all terms containing $\textbf{y}$ are reduced to $0$ mod $p^h$ we have that $$G(\textbf{x},G(\textbf{y},s(\textbf{x}))) \equiv G(\textbf{x},s(\textbf{x})) \ (\mbox{mod} \ p^h) \equiv 0 \ (\mbox{mod} \ p^h).$$ \end{proof}

\begin{defn} \label{germ} Let $\Gc$ be a $K$-Lie group. We denote by $\mathcal{O}^{la}(\mathcal{G})_e$, the \textit{ring of germs of locally analytic functions on $\Gc$ in $e$}. \end{defn} 

By Lemma \ref{open}, the ring of germs of locally analytic functions on $\Gc(0)$ in $e$ is given by $$\OGe= \varinjlim_h \mathcal{O}^{la}(\Gc(h), K).$$

\begin{defn}
The noetherian $R$-algebra $$R\{\underline{t}\}:=\{ f(\underline{t})=\sum_{\underline{j}} b_{\underline{j}} \underline{t}^{\underline{j}} \ | \ b_{\underline{j}} \in R, |b_{\underline{j}}| \rightarrow 0 \ \mbox{as} \ |\underline{j}| \rightarrow \infty\}$$ is called the \textit{algebra of strictly convergent power series over $R$}. 
\end{defn} 

\begin{rem} \label{rem1} Recall from non-archimedean analysis that every $f$ in $R[[\underline{t}]]$ converges on the open polydisc $\{z \in K^m \mid |\underline{z}| < 1 \}$ and every $f$ in $\mathfrak{m}[[\underline{t}]]$ converges on the closed polydisc $\{z \in K^m \mid |\underline{z}| \leq 1 \}$. The algebra $R\{\underline{t}\}$ is the sub-algebra of $R[[\underline{t}]]$ consisting of those power series which converge on $R^m$, since an infinite sum converges in a non-archimedean field if and only if its terms tend to zero. \end{rem}

\begin{defn} The $K$-algebra $$K\langle \underline{t}\rangle:=\{ f(\underline{t})=\sum_{\underline{j}} b_{\underline{j}} \underline{t}^{\underline{j}} \ | \ b_{\underline{j}} \in K, |b_{\underline{j}}| \rightarrow 0 \ \mbox{as} \ |\underline{j}| \rightarrow \infty\} $$ is the algebra of power series over $K$ which converge on $R^m$ in $K^m$ and is called \textit{Tate algebra}. The elements of $K\langle \underline{t}\rangle$ are called \textit{rigid analytic functions}. \end{defn}

Note that the convergence condition means that, for any $n \in \mathbb{N}$, there exists $j_0 \in \mathbb{N}$ such that for $|\underline{j}| > j_0$, the coefficient $b_{\underline{j}}$ belongs to $\pi^nR$, where $\pi$ is the uniformizing parameter, i.e. $(\pi)=\mathfrak{m}$. We have that  $K\langle \underline{t}\rangle = R\{\underline{t}\} \otimes_{R} K$, see \cite{nic}.

Since germs of locally analytic functions are none other than germs of rigid analytic functions, we can identify $\OGe$ with the limit of Tate algebras $$\OGe \cong \varinjlim_h K\langle p^{-h}\underline{t}\rangle.$$

\begin{defn} \label{defOe} The cobar complex $(\mathcal{O}^{la}(\mathcal{G}(0)^\bullet)_{e}, \partial)$ with $$\mathcal{O}^{la}(\mathcal{G}(0)^n)_e:= \varinjlim_h K\langle p^{-h}\underline{t}_{1 \ldots ,n}\rangle$$ and with the usual differential as in Definition \ref{gc} leads to cohomology groups $H^n(\OGe,K)$. 
\end{defn}

\begin{lemma} \label{limit}
The limit morphism $$\Phi_\infty: H^n(\mathcal{O}^{la}(\Gc(0)^\bullet)_e, K) \rightarrow H^n(\mathfrak{g}, K)$$ is an isomorphism. 
\end{lemma}

\begin{rem}
A proof of this lemma can also be found in \cite[Lemma 4.3.3]{HKN}, however our proof will be independent of the work of Lazard \cite{lazard} which is the utmost concern of this work. 
\end{rem}

\begin{defn} \label{Oh} Let $\Gc(h)$ be an $m$-standard group of level $h$. We denote by $\mathcal{O}_c(\Gc(h))$ the \textit{ring of convergent functions on $\Gc(h)$}, i.e. those formal power series in $R[[p^{-h}\underline{t}]]$ which are convergent on the closed polydisc $\{\underline{z} \in K^m \mid |\underline{z}| \leq p^{-h} \}$. \end{defn}

\begin{rem} \label{OZp}
Since $R\{\underline{t}\}= \{ f \in R[[\underline{t}]] \mid f \ \mbox{converges on} \ R^m\}$, due to Remark \ref{rem1}, we get for the ring of convergent functions on the $m$-standard group $\Gc(h)$ of level $h$ that $\mathcal{O}_c(\Gc(h))= R\{p^{-h}\underline{t}\}$.
\end{rem}

The following Lemma \ref{convlemma} will not only be of interest for the proof of Lemma \ref{limit} but also for the proof of the Main Theorem. We will see that functions of this modified ring of functions $\tOG$ still converge. 

\begin{lemma}[Lemma of Convergence] \label{convlemma} 
Let $G$ be a formal group law over $R$, let $\Gc(h)$ be the $m$-standard group of level $h$ to $G$ and let $\mathcal{O}_c(\Gc(h))$ be the ring of convergent functions on $\Gc(h)$. Let $G_h$ be the associated $m$-dimensional formal group law (see Proposition \ref{rel}) and let $$\tilde{\mathcal{O}}(G_h) = \left\lbrace  \sum_{\underline{j}} b_{\underline{j}} (p^{-h}\underline{t})^{\underline{j}}  \in K[[p^{-h}\underline{t}]] \ | \ \underline{j}! b_{\underline{j}} \in R \right\rbrace,$$ see Definition \ref{mO}, be its modified ring of functions. Then $$\tilde{\mathcal{O}}(G_h) \subset \mathcal{O}_c(\Gc(k))$$ for $k>k_0=h+\frac{1}{p-1}$. \end{lemma}
\begin{proof}
Let $f \in \tilde{\mathcal{O}}(G_h)$. Then $f$ can be written as $$f= \sum_{\underline{j}} b_{\underline{j}} (p^{-h}\underline{t})^{\underline{j}},$$ with $\underline{j}! b_{\underline{j}} \in R$. We claim that $f \in \mathcal{O}_c(\Gc(k))$ for $k>k_0=h+\frac{1}{p-1}$. To see this, let us rewrite $f$ in the following way:
\begin{align*} f = \sum_{\underline{j}} b_{\underline{j}} p^{-h|\underline{j}|} \underline{t}^{\underline{j}} =\sum_{\underline{j}} b_{\underline{j}} p^{(k-h)|\underline{j}|} (p^{-k}\underline{t})^{\underline{j}}.\end{align*}
We prove now that $ b_{\underline{j}} p^{(k-h)|\underline{j}|} \in \mathfrak{m}$. Since $\underline{j}! b_{\underline{j}} \in R$ we know that $$|b_{\underline{j}}| < p^{\frac{|\underline{j}|}{p-1}},$$ where we use that for a prime $p$ we have $v_p(n!) < \frac{n}{p-1}$, see \cite[Lemma 4.3.3]{gou}. Thus we can conclude that
\begin{align*} |b_{\underline{j}} p^{(k-h)|\underline{j}|}| &< p^{\frac{|\underline{j}|}{p-1}} p^{-(k-h)|\underline{j}|} \\ 
&= p^{\frac{|\underline{j}|}{p-1}} p^{-k|\underline{j}|} p^{h|\underline{j}|} \\ &< p^{\frac{|\underline{j}|}{p-1}} p^{-(h+\frac{1}{p-1})|\underline{j}|} p^{h|\underline{j}|} =1. \end{align*}
Using the observation of Remark \ref{rem1} we know that $f$ converges on the closed polydisc $\{z \in K^m \mid |\underline{z}| \leq p^{-k} \},$ i.e. $f \in \mathcal{O}_c(\mathcal{G}(k))$. \end{proof}

\begin{cor}\label{mod2}
The map $$\mathfrak{o}: \varinjlim_h \mathcal{O}_c(\Gc(h)) \rightarrow \varinjlim_h \tilde{\mathcal{O}}(G_h)$$ is an isomorphism.
\end{cor}
\begin{proof} The map $\mathfrak{o}$ is injective since $\mathcal{O}_c(\Gc(h)) \subset \tilde{\mathcal{O}}(G_h)$ and since the direct limit is exact. For surjectivity let $f$ be an element of $\varinjlim_h \tilde{\mathcal{O}}(G_h)$. Then there exists $h$ such that $f \in \tilde{\mathcal{O}}(G_h)$. However, after Lemma \ref{convlemma} $f \in \mathcal{O}_c(\mathcal{G}(k))$ for $k>k_0=h+\frac{1}{p-1}$.
\end{proof}

\begin{proof}[Proof of Lemma \ref{limit}] 
The proof is essentially the conjunction of Proposition \ref{gen} with all preceding results in this section. We already know that
\begin{align*} \OGe &= \varinjlim_h \mathcal{O}^{la}(\Gc(h), K) \\ &= \varinjlim_h K\langle p^{-h}\underline{t}\rangle \\&\stackrel{\ref{OZp}}{=} \varinjlim_h \mathcal{O}_c(\Gc(h)) \otimes_{R} K \\
&\stackrel{\ref{mod2}}{\cong} \varinjlim_h \tilde{\mathcal{O}}(G_h) \otimes_{R} K.\end{align*}
Let $\mathfrak{g}_h$ be the associated Lie algebra to $G_h$. Then we know from Proposition \ref{gen} that 
\begin{align*}\varinjlim_h \tilde{\mathcal{O}}(G_h) \otimes_{R} K \cong \varinjlim_h \U*(\mathfrak{g}_h, R) \otimes_{R} K.\end{align*} 
However, $\varinjlim_h \U*(\mathfrak{g}_h, R) \cong \U*(\mathfrak{g}, R) \otimes K$ and by Propositions \ref{propu}, \ref{k1prop} and \ref{k2prop} of Subsection \ref{cc} we obtain $$ H^n(\mathcal{O}^{la}(\Gc(0)^\bullet)_e, K) \cong H^n(\g, K).$$
\end{proof}

\subsection{Proof of the Comparison Theorem for standard groups} \label{secmthm}
We mentioned in Remark \ref{sketch} that injectivity of the map $$\Phi_s: H^n_{la}(\Gc(h), K) \rightarrow H^n(\g, K)$$ follows from a spectral sequence argument as in the proof of Theorem 4.3.1 in \cite{HKN}. Hence we will first prove the extistence of the required spectral sequence.

\begin{defn}  Let $\Gc$ be a $K$-Lie group and $\Hc$ a closed subgroup of $\Gc$. We define $I_K:=\Ind_{\Hc \rightarrow \Gc}^{la}(K)$ to be the space of locally analytic maps $f: \Gc \rightarrow K$ such that $f$ is $\Hc$-equivariant. \end{defn}

\begin{lemma}[Shapiro's Lemma] \label{shap} Let $\Gc$ be a $K$-Lie group and $\Hc$ a closed subgroup of $\Gc$. Then $$H^\bullet_{la}(\Gc, I_K) = H^\bullet_{la}(\Hc,K).$$
\end{lemma}
\begin{proof} See \cite[Prop. 3 (Shapiro's Lemma), Remark (2) and (3)]{casselm}. \end{proof}

The proof of the following theorem about the spectral sequence for locally analytic group cohomology can now, after we have seen that Shapiro's Lemma holds in the case of locally analytic group cohomology, be adopted from the proof for discrete groups, see e.g. \cite[Chap. II.1]{neukirch}.

\begin{thm} \label{ss} Let $\Gc$ be a $K$-Lie group and $\Hc$ be a closed normal subgroup of $\Gc$. Then there is a cohomological spectral sequence $$E_2^{pq}=H^p_{la}(\Gc/\Hc, H^q_{la}(\Hc, K)) \Rightarrow H^{p+q}_{la}(\Gc, K).$$ \end{thm}

\begin{cor}[Injectivity] \label{corinj1}
Let $G$ be a formal group law over $R$ and let $\Gc(0)$ be the $m$-standard group to $G$ with Lie algebra $\g \otimes_R K$. Then the map \begin{align*} \Phi_s: H^n_{la}(\Gc(0), K) \rightarrow H^n(\g, K) \end{align*} of Theorem \ref{thmstd} is injective. 
\end{cor}
\begin{proof} (Compare \cite[Cor. 4.3.4]{HKN}) Since all subgroups $\Gc(h)$ of $\Gc(0)$ are open, normal and of finite index (see Lemma \ref{open}), the spectral sequence of Theorem \ref{ss} degenerates to $$H^n_{la}(\Gc(0), K) \cong H^n_{la}(\Gc(h), K)^{\Gc(0)/\Gc(h)}.$$ Hence the restriction maps $$H^n_{la}(\Gc(0), K) \rightarrow H^n_{la}(\Gc(h), K)$$ are injective. As the system of open normal subgroups is filtered, this also implies that $$H^n_{la}(\Gc(0),K) \rightarrow \varinjlim_h H^n_{la}(\Gc(h), K)$$ is injective. We can therefore conclude injectivity of the map $\Phi_s$ from the injectivity of $\Phi_\infty$, since the cohomology functor commutes with the direct limit, i.e. in our case $\varinjlim_h H^n_{la}(\Gc(h)), K)=H^n(\OGe, K).$ \end{proof}

Note that the proof of injectivity uses that $K$ is locally compact, meaning that the proof can not be carried over to the case of the completion $\mathbb{C}_p$ of the algebraic closure $\overline{\mathbb{Q}}_p$ of $\Q_p$.

\begin{cor} \label{corinj2}
Let $G$ be a formal group law over $R$ and let $\Gc(h)$ be the $m$-standard group of level $h$ associated to $G$ with Lie algebra $\g$. Then the map \begin{align*} \Phi_s: H^n_{la}(\Gc(h), K) \rightarrow H^n(\g, K) \end{align*} of Theorem \ref{thmstd} is injective for all $h \geq 0$. 
\end{cor}
\begin{proof} Proposition \ref{rel} tells us that there exists an isomorphic formal group law $G_h$ to $G$ such that $\Gc(h)=\Gc_h(0)$. By Corollary \ref{corinj1} we know that $$H^n_{la}(\Gc(h), K)=H^n_{la}(\Gc_h(0), K) \rightarrow H^n(\g_h, K)$$ is injective, where $\g_h$ is the Lie algebra associated to $G_h$. However, $$\Hom_{R}({\bigwedge}^n\g_h, R) \otimes_R K \cong \Hom_{R}({\bigwedge}^n\g, R) \otimes_R K$$ so that we get injectivity of the map $\Phi_s$.\end{proof}

\begin{proof}[Proof of the Comparison Theorem for standard groups \ref{thmstd}] 
Since we already know from Corollary \ref{corinj2} that $$\Phi_s: H^n_{la}(\Gc(h), K) \rightarrow H^n(\g, K)$$ is injective for $h \geq 0$ we now concentrate on surjectivity. Let $[c]$ be a cohomology class in $H^n(\g, K)$. Then by Theorem \ref{phi} $[c]$ is represented by a cocyle \hbox{$\tilde{c} \in \tilde{\mathcal{O}}(G)^{\hat{\otimes}n} \otimes_{R} K$}. By Lemma \ref{convlemma} we know that $$\tilde{\mathcal{O}}(G) \subset \mathcal{O}_c(\Gc(h)) \subset \mathcal{O}^{la}(\Gc(h)), \ \mbox{for} \ h > h_0=\frac{1}{p-1}.$$ Hence $\Phi_s$ is surjective for all $h > h_0=\frac{1}{p-1}$.
\end{proof}

\end{document}